\documentclass[12pt,reqno]{amsart}
\pdfoutput=1
\usepackage[headings]{fullpage}
\usepackage{amssymb,amsmath,bbm,tikz}
\usepackage{graphicx,color,enumerate, float}
\usepackage[all,cmtip]{xy}
\usepackage{url}
\usepackage[margin=1.2in]{geometry}
\usepackage{verbatim}
\usepackage{pb-diagram}

\usepackage[bookmarks=true,%
    colorlinks=true,%
    linkcolor=blue,%
    citecolor=blue,%
    filecolor=blue,%
    menucolor=blue,%
    urlcolor=blue,%
    breaklinks=true]{hyperref}

\newtheorem{theorem}{Theorem}[section]
\theoremstyle{definition}

\newtheorem{lemma}[theorem]{Lemma}
\newtheorem{definition}[theorem]{Definition}
\newtheorem{remark}[theorem]{Remark}

\renewcommand{\setminus}{{\smallsetminus}}

\def\be{\begin{equation}}
\def\ee{\end{equation}}

\graphicspath{{figures/}}

\begin{document}

\title[The well rounded deformation retraction]{Understanding the well-rounded deformation retraction of Teichm\"uller space} 

\author{Ingrid Irmer}
\address{SUSTech International Center for Mathematics\\
Southern University of Science and Technology\\Shenzhen, China
}
\address{Department of Mathematics\\
Southern University of Science and Technology\\Shenzhen, China
}
\email{ingridmary@sustech.edu.cn}


\begin{abstract} In \cite{Tspine} it was shown that there is a mapping class group-equivariant deformation retraction of the Teichm\"uller space of a closed surface onto a CW complex with dimension equal to the virtual cohomological dimension of the mapping class group. This paper studies the image of this deformation retraction, and shows that when the analogy with the well-rounded deformation retraction of $SL(n,\mathbb{Z})$ is defined correctly via a notion of duality, this deformation retraction is analogous to the well-rounded deformation retractions of \cite{Ash}, \cite{Soule} and \cite{Voronoi}. In the process, an elementary necessary condition is derived for a cycle in the geometric realisation of Harvey's curve complex to represent a nontrivial homology class.
\end{abstract}

\maketitle

{\footnotesize
\tableofcontents
}


\section{Introduction}
\label{sec.intro}

Denote by $\mathcal{S}_{g}$ a closed, connected, orientable topological surface of genus $g\geq 2$, and let $\Gamma_{g}$ be the mapping class group of $\mathcal{S}_{g}$. The Thurston spine $\mathcal{P}_{g}$ is the subset of the Teichm\"uller space $\mathcal{T}_{g}$ of $\mathcal{S}_{g}$ corresponding to the marked hyperbolic surfaces for which the complement of the set of shortest geodesics (the systoles) consists of a disjoint union of polygons, i.e. $\mathcal{P}_{g}$ is the set of points of $\mathcal{T}_{g}$ at which the systoles fill the surface.  $\mathcal{P}_{g}$ is compact modulo the action of $\Gamma_{g}$ and is the image of a $\Gamma_{g}$-equivariant deformation retraction of $\mathcal{T}_{g}$, \cite{ThurstonSpineSurvey}, \cite{Thurston}.\\

The term ``well-rounded deformation retraction'' was introduced by Ash in the context of an action of $SL(n,\mathbb{Z})$ on the space of $n\times n$ positive-definite real symmetric matrices. In this paper, an analogue of these well-rounded deformation retractions will be introduced for the action of $\Gamma_{g}$ on $\mathcal{T}_{g}$, demonstrating the striking analogy between various types of groups discussed in the survey articles \cite{BestvinaSurvey} and \cite{BV}. In order to make this definition, it is essential to understand the role played by duality implicit in the construction of equivariant deformation retractions of the spaces with group actions surveyed in \cite{BV}. 

\begin{theorem}
\label{whichsubcomplex}
For every $g\geq 2$ there is a well-rounded deformation retraction of $\mathcal{T}_{g}$ onto a CW complex of dimension $4g-5$.
\end{theorem}

Theorem \ref{whichsubcomplex} is proven by constructing duals to $\mathcal{P}_{g}$ and to complexes obtained as  retracts of $\mathcal{P}_{g}$. A dual is labelled by a set of curves. The homological spans of sets of curves labelling the duals is related to collapsibility properties of the complex. This makes use of the embedding of Harvey's complex of curves in the boundary of the thick part of Teichm\"uller space. Dual cells are also implicit in the construction of \cite{Ash} and in the construction of an equivariant deformation retraction of the Teichm\"uller space of punctured surfaces, \cite{Harer}, \cite{PennerComplex}. In the former case dual cells are labelled by sets of vectors, and in the latter case by fatgraphs, which determine sets of curves as explained in \cite{PM}.\\

One might hope that a theorem analogous to Theorem \ref{whichsubcomplex} holds for the action of $Out(F_{n})$ on Outer space, once a topologically meaningful definition of well-rounded deformation retraction is found. This would involve deciding on an analogue of the curve complex, such as the free factor complex, \cite{HM}, \cite{HV}. \\

There are many classical examples of closed hyperbolic surfaces contained in the Thurston spine, such as the Bolza surface in genus 2, the Klein quartic in genus 3, and the examples in arbitrary genus given in \cite{SchmutzMaxima} and \cite{SchmutzMorse}. All these well-known classical hyperbolic surfaces have the property that the systoles not only fill the surface, but also span the rational homology of the surface. In \cite{Tspine} it was shown that there exists a $\Gamma_{g}$-equivariant deformation retraction of $\mathcal{T}_{g}$ onto a complex of dimension $4g-5$ contained in the Thurston spine. A goal of this paper is to understand what cells of a subdivision of $\mathcal{P}_{g}$ are retained in such a subcomplex. \\

A $\Gamma_{g}$-equivariant construction of duals was explained in detail in \cite{STD}, and this construction is surveyed in Section \ref{dualitysec}. Informally, a well-rounded deformation retraction is a $\Gamma_{g}$-equivariant deformation retraction, with the property that the set of curves labelling each dual not only fills $\mathcal{S}_{g}$, but also spans $H_{1}(\mathcal{S}_{g};\mathbb{Q})$. There is a canonical bijection between the set of locally top-dimensional cells and their duals. However, for closed surfaces of genus greater than one, the homological span of a set of curves labelling a dual is \textit{not} always identical to that of the set of systoles at the critical point; the former is typically larger. This is also seen in the case of punctured surfaces, where the set of curves labelling a cell is calculated using the screens of McShane and Penner, \cite{PM}.\\

In \cite{bourque2023failure} an attempt was made to draw an analogy with \cite{Ash} by making the simplest (from a purely linguistic --- not a mathematical --- point of view) definition. As pointed out in \cite{bourque2023failure}, this naive approach cannot work due to a corollary of previous work by the author of \cite{bourque2023failure}. There are a number of other elementary reasons why the definition from \cite{bourque2023failure} could not work. The ``failed analogy'' was given in \cite{bourque2023failure} as an excuse for an attack on Thurston, via the unsubstantiated claim that the analogy was the motivation for the construction in \cite{Thurston}.\\


In \cite{MS} it was shown that Thurston's deformation retraction can be extended to an equivariant deformation retraction onto a complex $\mathcal{P}^{X}_{g}\subset \mathcal{P}_{g}$ consisting of a choice of unstable manifolds of the systole function, determined by a vector field $X$. The choice was made to work with $\mathcal{P}^{X}_{g}$ in place of $\mathcal{P}_{g}$. The reason for this is that it makes the definition of duality more intuitive; a dual to a locally top-dimensional cell of $\mathcal{P}^{X}_{g}$ is essentially a choice of stable manifold of a critical point. Note that, although stable and unstable manifolds of critical points of topological Morse functions were defined in \cite{Morse}, unlike in the smooth case, they are not defined uniquely. As explained in \cite{STD}, this is one advantage of replacing these concepts with the canonical objects $\mathcal{P}_{g}$ and corresponding ``dual'' sets of minima.\\

The number $4g-5$ in Theorem \ref{whichsubcomplex} is significant, in that it was shown to be the virtual cohomological dimension of the mapping class group, \cite{Harer}, and hence gives a lower bound on the dimension of every equivariant spine. The complexes $\mathcal{P}_{g}$ and $\mathcal{P}_{g}^{X}$ often have dimension greater than $4g-5$ and are not in general well-rounded. As discussed in \cite{Ni}, $g=5$ is presumably the smallest genus for which the dimension of $\mathcal{P}_{g}$ is greater than $4g-5$ and for which $\mathcal{P}_{g}$ is not well-rounded. As explained in \cite{BV}, one of the reasons it was initially conjectured that there exists an equivariant spine of $\mathcal{T}_{g}$ of dimension $4g-5$ is an analogy between different families of groups, $Out(F_{n})$, $GL(n,\mathbb{Z})$ and mapping class groups of orientable surfaces of genus at least 1. \\

In the context of this analogy, Harvey's definition of the curve complex $\mathcal{C}_{g}$ of $\mathcal{S}_{g}$ and its relationship with $\mathcal{T}_{g}$ was motivated by Tits buildings for symmetric spaces. It was shown in \cite{Ivanov} that $\mathcal{C}_{g}$ is $\Gamma_{g}$-equivariantly homotopy equivalent to the boundary of the thick part of $\mathcal{T}_{g}$. Informally this is a consequence of the fact that the set of systoles at any point in the thin part of $\mathcal{T}_{g}$ is a multicurve, for which the corresponding stratum can be identified with a simplex of Harvey's curve complex.\\

As explained in \cite{STD}, one way of constructing duals to cells of $\mathcal{P}^{X}_{g}$ is by gluing together examples of Schmutz Schaller's sets of minima. This is surveyed in Section \ref{dualitysec}. A set of minima $\mathrm{Min}(C)$ depends on a set $C$ of filling curves, and will be defined in Section \ref{defns}. As explained in \cite{SchmutzVoronoi}, certain sets of minima can be understood as a non-Euclidean analogue of Voronoi's cells defined on the space of positive definite quadratic forms, \cite{Voronoi}. An English introduction to Voronoi's techniques and related constructions is given in \cite{Gangl}.\\

As defined, well-rounded deformation retractions of $\mathcal{T}_{g}$ could only be unique up to ambient isotopy. This has to do with the fact that there are choices made in the construction of $\mathcal{P}^{X}_{g}$; any argument using duals or topological Morse functions can only determine the topological properties of a spine, not specify the geometry of a canonical choice. Questions about minimality in the sense of \cite{PS} of well-rounded deformation retractions are closely tied up with the question of whether the converse to the next lemma holds.\\

The horizon map, defined in \cite{STD} and reviewed in Section \ref{dualitysec}, takes a set of minima $\mathrm{Min}(C)$ to a subcomplex of the barycentric subdivision $\mathcal{C}_{g}^{\circ}$ of $\mathcal{C}_{g}$ representing a bordification of $\mathrm{Min}(C)$. This subcomplex contains all vertices labelled by multicurves that can be made arbitrarily short within $\mathrm{Min}(C)$. Theorem \ref{whichsubcomplex} is based on the next lemma.

\begin{lemma}
\label{oneway}
Suppose $D$ is a union of sets of minima, for example, $D$ is a dual to a locally top-dimensional cell of $\mathcal{P}^{X}_{g}$.  If the homology classes of the set of curves that can be made arbitrarily short on $D$ (the curves labelling the dual) do not span $H_{1}(S_{g};\mathbb{Q})$ then the horizon map $h$ takes $D$ to a boundary in $\mathcal{C}^{\circ}_{g}$.
\end{lemma}

The result in \cite{matroid} strongly suggests a partial converse to Lemma \ref{oneway}. While the converse to Lemma \ref{oneway} appears intuitively reasonable, the author is not aware of any techniques available for proving it. Somewhat similar statements were proven in Section 3 of \cite{MS}, but these arguments require control over geometric intersection numbers, not algebraic intersecton numbers of curves.\\


As mentioned above, the groups $Out(F_{n})$, $GL(n,\mathbb{Z})$ and mapping class groups of punctured surfaces all act on spaces for which there exist equivariant deformation retractions onto spines. For surfaces without punctures the analogue of the piecewise-linear structures on the spaces become piecewise-smooth. When working with the sets of minima, a number of new phenomena, most importantly the transversality issues ``breakdown in regularity'' documented in \cite{SchmutzMorse} and the ``unbalanced strata'' defined in \cite{MS}, mean that the analogue of Voronoi's techniques in this case might not give an equivariant cell decomposition, but at best an equivariant ``pinched cell decomposition'' (defined on page 7 of \cite{STD}). How to construct such an object is outlined at the end of \cite{STD}. Conditions ensuring that an analogue of Voronoi's cell decomposition exists are given in \cite{SchmutzVoronoi}. More work is needed on computing such objects and determining just how degenerate they can be.\\


\textbf{Outline of the paper.} Section \ref{defns} defines the basic concepts and provides the background knowledge that will be used throughout this paper. The notion of duality is defined in Section \ref{dualitysec}; this is necessary because the objects with which this paper is concerned are not the usual embedded submanifolds. It is explained how duals are used to relate the topology of $\mathcal{P}_{g}$ with the topology of Harvey's complex of curves, where the latter is viewed as describing a bordification of $\mathcal{T}_{g}$. Well-rounded deformation retractions of $\mathcal{T}_{g}$ are defined in Section \ref{fillandspansub} where the new theorems of this paper are also proven.


\subsection*{Acknowledgements} 
The author would like to thank the IHES for its hospitality while some of this work was being done, and to S. Garoufalidis for many helpful comments.


\section{Assumptions and Background}
\label{defns}

The purpose of this subsection is to provide definitions and background that will be used throughout this paper.\\

All surfaces are closed, compact, connected and orientable without marked points. Subsurfaces are embedded with homotopically nontrivial boundary curves. The symbol $S_{g}$, $g\geq 2$, will be used to denote a topological surface without boundary endowed with a marked hyperbolic structure corresponding to a point in $\mathcal{T}_{g}$. When there is no possibility for confusion, the same symbol $S_{g}$ will also be used to denote the topological surface of genus $g$ without boundary. Curves on surfaces are assumed to be unoriented nontrivial isotopy classes of embeddings of $S^{1}$ into $\mathcal{S}_{g}$. Sometimes a symbol for a curve will be used interchangeably to represent the image of a particular representative of the isotopy class, for example a geodesic. A multicurve is a set of curves, with pairwise geometric intersection number zero.\\

The mapping class group of $S_{g}$ will be denoted by $\Gamma_{g}$.\\

The length of a curve $c$ on a marked hyperbolic surface defines a smooth function $L(c):\mathcal{T}_{g}\rightarrow\mathbb{R}_{+}$. The function $L(c)$ is a special case of a length function.

\begin{definition}[Length function]
A finite ordered set of curves $C=(c_{1}, \ldots, c_{n})$ together with an ordered set of real, positive weights $A=(a_{1}, \ldots, a_{n})$ define a smooth function $L(A,C):\mathcal{T}_{g}\rightarrow \mathbb{R}^{+}$ as follows:
\begin{equation*}
L(A,C)(x)\,=\, \sum_{j=1}^{n} a_{j}L(c_{j})(x)
\end{equation*}
Any such function $L(A,C)$ will be called a length function.
\end{definition}
Length functions satisfy several convexity properties, as shown in \cite{Bestvina}, \cite{Kerckhoff}, and \cite{Wolpert}.  For example, they are strictly convex along Weil-Petersson geodesics.\\

A set $C$ of curves is said to fill the surface $S_{g}$ when cutting $S_{g}$ along the geodesic representatives of the curves in $C$ gives a set of polygons.

\begin{definition}
The Thurston spine, $\mathcal{P}_{g}$, is the set of points in $\mathcal{T}_{g}$ at which the set of shortest curves on $S_{g}$ (the systoles) fill $S_{g}$.
\end{definition}

It follows from a theorem due to Lojasiewicz, \cite{Lojasiewicz1964}, that $\mathcal{P}_{g}$ is a CW complex. Both $\mathcal{P}_{g}$ and $\mathcal{T}_{g}$ can be decomposed into sets, each of which has systoles given by a fixed set of curves. Following Thurston, these sets will be called strata. A stratum with set of systoles $C$ will be denoted by $\mathrm{Sys}(C)$. Top-dimensional strata of $\mathcal{T}_{g}$ are open sets labelled by a single curve. The cells of $\mathcal{P}_{g}^{X}$ come from the unstable manifolds of critical points of the systole function. Due to the fact that the systole function is not smooth, there are some technicalities involved in obtaining a cell decomposition of $\mathcal{P}_{g}^{X}$ in this way. These are discussed in \cite{Tspine}.\\

The systole function is a continuous, piecewise smooth function from $\mathcal{T}_{g}$ to $\mathbb{R}_{+}$, whose value at any point $x$ of $\mathcal{T}_{g}$ is given by the lengths of the systoles at that point. When restricted to a top-dimensional stratum of $\mathcal{T}_{g}$, the systole function is smooth. A useful property of the systole function is that it is invariant under the action of $\Gamma_{g}$.\\

As is the case for most cell complexes, it might not be true that every point of $\mathcal{P}_{g}$ has a neighbourhood intersecting a cell of the same maximal dimension. A locally top-dimensional cell is a cell that is not on the boundary of a larger dimensional cell.\\

\begin{definition}[Topological Morse function]
A topological Morse function $f$ is a continuous real-valued function on a topological $n$-dimensional manifold $M$, with the property that the points of $M$ are all either regular points or critical points. A regular point $p\in M$ is a point with an open neighbourhood $U$ in $M$, such that $U$ is a homeomorphic coordinate patch, with one of the coordinates on $U$ being the function $f$. When $p$ is a critical point, there exists a neighbourhood $U$ of $p$, and an index, $k\in \mathbb{Z}$, $0\leq k\leq n$, such that $U$ is a homeomorphic coordinate patch with the coordinates $\{x_{1}, \ldots, x_{n}\}$, and in $U$, $f$ is given by the formula
\begin{equation*}
f(x)-f(p)=\sum_{i=1}^{i=n-k}x_{i}^{2} - \sum_{i-n-k+1}^{i=n}x_{i}^{2}
\end{equation*}
\end{definition}

The systole function is a topological Morse function, \cite{Akrout}, \cite{SchmutzMorse}. Topological Morse functions can be used like the usual smooth Morse functions when working with homology and deformation retractions. The critical points of the systole function are all contained in $\mathcal{P}_{g}$, \cite{Thurston}.\\





Choose $\delta>0$ less than or equal to the Margulis constant $\epsilon_{M}$. The $\delta$-thick part of $\mathcal{T}_{g}$ will be denoted by $\mathcal{T}_{g}^{\delta}$. The set $\mathcal{T}_{g}^{\delta}$ is the complement of the pre-image of $(0, \delta)$ under the systole function. Let $\mathcal{C}_{g}^{\circ}$ be the barycentric subdivision of Harvey's curve complex $\mathcal{C}_{g}$. \\

\begin{theorem}[\cite{Ivanov}]
\label{complexembedding}
$\mathcal{C}_{g}^{\circ}$ is $\Gamma_{g}$-equivariantly homotopy equivalent to $\partial\mathcal{T}_{g}^{\delta}$.
\end{theorem}

Note that the dimension of $\mathcal{C}_{g}$ is $3g-4$; less than the dimension $6g-7$ of $\partial\mathcal{T}_{g}^{\delta}$.\\


\textbf{Schmutz Schaller's sets of minima.} Sets of minima were introduced in \cite{SchmutzMorse}. 

\begin{definition}[Set of minima $\mathrm{Min}(C)$ and $\partial\mathrm{Min}(C)$]
Let $C$ be a set of curves on $\mathcal{S}_{g}$. The set of minima, $\mathrm{Min}(C)$, is the set of all $p\in \mathcal{T}_{g}$ such that every derivation $v\in T_{p}\mathcal{T}_{g}$ has the property that either
\begin{itemize}
\item{there exists $c_{i}, c_{j}\in C$ such that $vL(c_{i})(p)>0>vL(c_{j})(p)$}
\item{$vL(c_{i})(p)=0$ for every $c_{i}$ in $C$.}
\end{itemize}
Alternatively, $\mathrm{Min}(C)$ is the set of all points in $\mathcal{T}_{g}$ at which $L(A,C)$ realises its minimum for some $A\in \mathbb{R}_{+}^{|C|}$.
\end{definition}

The equivalence of the two definitions follows from the observation that length functions are convex, so a necessary and sufficient condition for a length function $L(A,C)$ to have a minimum at $x$ is that the gradient at $x$ is zero. It was shown in Lemma 1 of \cite{SchmutzMorse} that when the curves in $C$ fill, $L(A,C)$ has a unique minimum for every $A$ with strictly positive entries. Hence, by proposition 1 of \cite{Thurston}, $\mathrm{Min}(C)$ is nonempty iff $C$ fills.\\

This paper will be exclusively concerned with sets of minima $\mathrm{Min}(C)$ for which the curves in $C$ fill and have pairwise geometric intersection number at most one. This rules out nonsimple filling curves or separating curves in the set $C$.\\

There is a sense in which Schmutz Schaller's sets of minima and the Thurston spine are dual to each other. This is explored in detail in \cite{STD}. There are qualifications and technical details needed to make this rigorous, but informally, $\mathcal{P}_{g}$ can be thought of as the unstable manifolds of the critical points, and the stable manifolds are corresponding sets of minima dual to $\mathcal{P}_{g}$. The stable and unstable manifolds of a topological Morse function were defined in \cite{Morse}. Due to the fact that the systole function is not smooth, the stable and unstable manifolds of critical points of the systole function may not be uniquely defined. This is one motivation for working instead with $\mathcal{P}_{g}$ and sets of minima.




\section{Duality}
\label{dualitysec}
This section begins with a discussion of what is meant by a dual to $\mathcal{P}_{g}$ or to a retract of $\mathcal{P}_{g}$, and how sets of minima can be used to define these duals. First of all, motivating examples will be given, demonstrating existence, followed by a definition. The horizon map from \cite{STD} will then be defined. This map relates a set of minima to a subcomplex of the barycentric subdivision $\mathcal{C}^{\circ}_{g}$ of Harvey's curve complex $\mathcal{C}_{g}$. A reference for this section is Section 5 of \cite{STD}.\\
 
 A locally top-dimensional cell of $\mathcal{P}_{g}^{X}$ is contained in an unstable manifold of a critical point $p$, with the property that any critical points on the boundary of this cell have strictly larger index than $p$. An example of a dual to $\mathcal{P}_{g}^{X}$ at $p$ is a stable manifold of the critical point $p$, provided this stable manifold only intersects $\mathcal{P}_{g}^{X}$ in the point $p$. \\
 
As explained in \cite{MS}, there is an equivariant deformation retraction of $\mathcal{T}_{g}$ onto $\mathcal{P}_{g}^{X}$ obtained from a systole function-increasing flow. This deformation retraction is an extension of Thurston's deformation retraction. One way of finding a stable manifold of $p$ is therefore to take the pre-image of $p$ under the deformation retraction of $\mathcal{T}_{g}$ onto $\mathcal{P}_{g}$ or $\mathcal{P}_{g}^{X}$. For critical points in locally top-dimensional cells of $\mathcal{P}_{g}^{X}$, this pre-image intersects $\mathcal{P}_{g}^{X}$ only in the point $p$. \\

Denote by $\overline{\mathcal{T}}_{g}$ the metric completion of $\mathcal{T}_{g}$ with respect to the Weil-Petersson metric, where $\partial \mathcal{T}_{g}:=\overline{\mathcal{T}}_{g}\setminus \mathcal{T}_{g}$.

\begin{theorem}[Theorem 1.2 of \cite{STD}]
\label{1point2}
Suppose $p$ is a critical point of the systole function with set of systoles $C$, and $T_{p}\mathrm{Min}(C)$ does not have any vectors in the tangent cone to $\mathcal{P}_{g}$. Then $\mathrm{Min}(C)$ is a cell with boundary in $\partial \mathcal{T}_{g}$. In addition, there is a homotopy of $\mathrm{Min}(C)$ to a cell obtained as the pre-image of $p$ under Thurston's equivariant deformation retraction of $\mathcal{T}_{g}$ onto $\mathcal{P}_{g}$. This homotopy fixes $p$ and keeps the thin part of $\mathrm{Min}(C)$ in the thin part of $\mathcal{T}_{g}$.
\end{theorem}
  
Note that $\mathcal{P}_{g}$ is contained in the thick part of $\mathcal{T}_{g}$, whereas the boundary of the set $\mathrm{Min}(C)$ from Theorem \ref{1point2} consists of points in the bordification of $\mathcal{T}_{g}$ corresponding to noded surfaces.\\
 
A set of minima $\mathrm{Min}(C)$ satisfying the conclusions of Theorem \ref{1point2} for some $C$ is an example of a dual to $\mathcal{P}_{g}$. The image of a dual to $\mathcal{P}_{g}$ (or $\mathcal{P}_{g}^{X}$, or a retract of $\mathcal{P}_{g}^{X}$) under a homotopy that keeps the thin part of the dual in the thin part of $\mathcal{T}_{g}$, does not introduce points of intersection with $\mathcal{P}_{g}$ (or $\mathcal{P}_{g}^{X}$, or a retract of $\mathcal{P}_{g}^{X}$), and keeps the unique point of intersection with $\mathcal{P}_{g}$ (or $\mathcal{P}_{g}^{X}$, or a retract of $\mathcal{P}_{g}^{X}$) in the same locally-top-dimensional cell of $\mathcal{P}_{g}$ (or $\mathcal{P}_{g}^{X}$, or a retract of $\mathcal{P}_{g}^{X}$) will also be called a dual to $\mathcal{P}_{g}$ (or $\mathcal{P}_{g}^{X}$, or a retract of $\mathcal{P}_{g}^{X}$).\\
 
When $p$ is contained in a locally top-dimensional cell of $\mathcal{P}_{g}^{X}$, it is possible that the condition in Theorem \ref{1point2} that $T_{p}\mathrm{Min}(C)$ does not have any vectors in the tangent cone to $\mathcal{P}_{g}$ breaks down. There might be examples in which there exist strata in $\mathcal{P}_{g}$ ``below'' the critical point $p$ in $\mathrm{Min}(C)$, that are not contained in the unstable manifold of any critical point. Also, performing an equivariant deformation retraction on $\mathcal{P}_{g}^{X}$ creates new locally top-dimensional cells. Section 5.2 of \cite{STD} explains how to construct duals to these locally top-dimensional cells by taking unions of sets of minima. In pathological cases, these unions could be quite complicated geometrically and may not be cells.  \\

The key defining properties of duals that will be used in the rest of this paper are
\begin{itemize}
\item{A dual intersects $\mathcal{P}_{g}^{X}$ or a complex obtained as a deformation retraction of $\mathcal{P}_{g}^{X}$ in a single point $q$.}
\item{The point $q$ is contained in a locally top-dimensional cell of the respective complex.}
\item{A homotopy keeping the thin part of the dual in the thin part of $\mathcal{T}_{g}$ can only take the dual to a set disjoint from the respective complex if the complex has nonempty boundary.}
\end{itemize}
  
It will only be necessary to consider duals that intersect an unstable manifold of a critical point $p$ in the critical point $p$.\\
 
Denote by $\mathcal{C}_{g}^{\circ}$ the barycentric subdivison of Harvey's complex of curves $\mathcal{C}_{g}$. The map 
$$h: \{\mathrm{Min}(C)\ |\ C \text{ fills and elements have pairwise intersection number at most one}\}\rightarrow \mathcal{C}_{g}^{\circ}$$ is called the horizon map in \cite{STD}. It maps a set $\mathrm{Min}(C)$ to the subcomplex of $\mathcal{C}_{g}^{\circ}$ with vertices labelled by multicurves that become arbitrarily short on $\mathrm{Min}(C)$. The properties of the horizon map were studied in Section 3 of \cite{STD}.\\
 
For $\delta\leq \epsilon_{M}$, the systoles on $\partial \mathcal{T}_{g}^{\delta}$ are pairwise disjoint. Consequently, each stratum of $\mathrm{Min}(C)\cap \partial \mathcal{T}_{g}^{\delta}$ determines a cell of $\mathcal{C}_{g}^{\circ}$. The image of the horizon map is homotopy equivalent to the induced cell decomposition of $\mathrm{Min}(C)\cap \partial \mathcal{T}_{g}^{\delta}$ in the limit as $\delta\rightarrow 0$. \\

The set of curves that determine vertices of $h(\mathrm{Min}(C))$ will be denoted by $h(\mathrm{Min}(C))_{v}$. As it is always possible to construct a dual from a union of sets of minima, in this paper, all duals to $\mathcal{P}_{g}^{X}$ and to complexes contained in $\mathcal{P}_{g}^{X}$ can and will be assumed to be a union of sets of minima. It therefore makes sense to talk of the image of a dual $D$ under the horizon map $h$. If a dual $D$ intersects $\mathcal{P}_{g}^{X}$ or a retract of $\mathcal{P}_{g}^{X}$ in a critical point with set of systoles $C$, it follows from Corollary 3.4 of \cite{STD} that the set of systoles at the critical point is contained in $h(D)_{v}$. In addition, $h(\mathrm{Min}(C))_{v}$ also contains the multicurves on the boundaries of subsurfaces filled by the geodesic representatives of subsets of $C$.\\

The set $h(D)_{v}$ will also be called the set of curves labelling the dual $D$.  These labels are conjectured to be unique, but this will not be needed here. \\

The horizon map is reminiscent of the ``screens'' defined in \cite{PM}. In a cell decomposition of decorated Teichm\"uller space, there are well-known cell decompositions, with the property that each cell is labelled by a fatgraph. A screen is a combinatorial object used to determine which curves can be made arbitrarily short on the cell labelled by a given fatgraph. Screens construct sets of curves with the same closure properties as the set $h(\mathrm{Min}(C))_{v}$; any curve on the boundary of a subsurface filled by a subset is also in the set. The notion of ``subsurface filled by'' is independent of whether the curves are contained in a fatgraph or the surface $S_{g}$.\\

\section{Well-rounded deformation retractions}
\label{fillandspansub}
This section begins with an important definition.

\begin{definition}[Well-rounded deformation retraction]
A well-rounded deformation retraction of $\mathcal{T}_{g}$ is an equivariant deformation retraction of $\mathcal{T}_{g}$ onto a CW-complex $\mathcal{W}_{g}\subset\mathcal{P}^{X}_{g}\subset \mathcal{P}_{g}$ for which every locally top-dimensional cell of $\mathcal{W}_{g}$ has a dual labelled by a set of curves that spans $H_{1}(\mathcal{S}_{g};\mathbb{Q})$.
\end{definition}

\begin{remark}
\label{subcomplex}
\textit{Suppose} $\mathcal{W}_{g}$ \textit{is the image of }$\mathcal{T}_{g}$ \textit{under a well-rounded deformation retraction and }$\mathcal{W}_{g}$ \textit{is not minimal, i.e. a further equivariant deformation retraction can be performed on} $\mathcal{W}_{g}$ \textit{to obtain a complex} $\mathcal{W}'_{g}$.\textit{ Then the duals to any newly created locally top-dimensional cells of}$\mathcal{W}'_{g}$ \textit{contain the duals to cells of} $\mathcal{W}_{g}$. \textit{It follows that} $\mathcal{W}'_{g}$ \textit{is also the image of a well-rounded deformation retraction. Informally, lower dimensional cells have higher dimensional duals, labelled by larger sets of curves.\\}

\textit{As discussed briefly in the introduction, well-rounded deformation retractions can be shown to be minimal if the converse to Lemma \ref{oneway}, restated below, holds.}
\end{remark}

The next lemma contains the key ideas of this paper.

\begin{lemma}[Lemma \ref{oneway} from the Introduction]
Suppose $D$ is a union of sets of minima, for example, $D$ is a dual to a locally top-dimensional cell of $\mathcal{P}^{X}_{g}$.  If the homology classes of the set of curves that can be made arbitrarily short on $D$ (the curves labelling the dual) do not span $H_{1}(S_{g};\mathbb{Q})$ then the horizon map $h$ takes $D$ to a boundary in $\mathcal{C}^{\circ}_{g}$.
\end{lemma}
\begin{proof}
In the following, the same symbol $D$ will be used to refer to both an embedding and the image in $\mathcal{T}_{g}$ of the embedding.\\

Suppose the curves in $h(D)_{v}$ do not span $H_{1}(S_{g};\mathbb{Q})$. For some $0<\delta<\epsilon_{M}$, a homotopy of $D$ fixing the points $D\cap \mathcal{T}_{g}^{\delta}$ and taking $D$ into $\mathcal{T}_{g}^{\delta}$ will be constructed. Informally the idea is to use a missing homology class to determine a preferred, nonvanishing direction in which the homotopy shifts points to eventually increase the diameter of the surface and hence reach the thin part of $\mathcal{T}_{g}$.\\

Choose an orientation on the curves in $h(D)_{v}$. Denote by $\tilde{S}_{g(M)}$ an $M$-sheeted cyclic cover of $S_{g}$ of genus $g(M)$ to which all the curves in $h(D)_{v}$ lift. Such a cover exists because the curves in $h(D)_{v}$ do not span homology. Let $\alpha$ be an oriented curve on $S_{g}$ with algebraic intersection number zero with each of the curves in $h(D)_{v}$. A fundamental domain of the cover can be constructed by cutting $S_{g}$ along $\alpha$.

\begin{remark}
\textit{The cyclic cover }$\tilde{S}_{g(M)}$ \textit{of} $S_{g}$\textit{ illustrates the type of pathological examples that can occur when a set of systoles fill the surface but does not span  rational homology. A concrete example with $g=5$ is given in \cite{Ni}. Suppose }$C$\textit{ is the set of systoles at the critical point at which a dual }$D$\textit{ intersects }$\mathcal{P}^{X}_{g}$. \textit{The lengths of curves in }$\tilde{S}_{g(M)}$\textit{ are no less than the lengths of their projections to }$S_{g}$.\textit{ The stratum }$\mathrm{Sys}(C)$ \textit{of} $\mathcal{T}_{g}$ \textit{determines a stratum }$\mathrm{Sys}(\tilde{C})$ \textit{in} $\mathcal{T}_{g(M)}$\textit{, where the curves in }$\tilde{C}$\textit{ are the connected components of the pre-images of the curves in }$C$. \\

\textit{The curves in }$\tilde{C}$\textit{ fill }$\tilde{S}_{g(M)}$\textit{, but for large index covers, a surface corresponding to a point in }$\mathrm{Sys}(\tilde{C})$\textit{ looks more like a graph than a surface that one would expect to see in the Thurston spine, where one expects the existence of a fundamental domain that is almost circular. By choosing }$M$\textit{ large enough, examples can be constructed in which the systoles fill, but the lengths of the systoles are arbitrarily far from a maximal value described in \cite{BuserSarnak} that increases approximately logarithmically with the genus of the surface.}
\end{remark}

By making $M$ large, the diameter of $\tilde{S}_{g(M)}$ can be made large. Suppose $\tilde{x}$ is in the thick part of $\mathcal{T}_{g(M)}$, and represents the cover of a surface corresponding to a point in $D$. The thick part of $D$ is compact, so any arguments that require $M$ to be sufficiently large can be made to work for any choice of $\tilde{x}$. The diameter $d(\tilde{x})$ of the surface represented by $\tilde{x}$ is realised by the length of an arc passing through roughly half of the fundamental domains of the cover, where fundamental domains as assumed to be obtained as connected components of $\tilde{S}_{g(M)}$ cut along the geodesic representative of $\tilde{\alpha}$. Call the endpoints of the arc $\tilde{p}_{1}$ and $\tilde{q}_{1}$. This arc is not unique. As the covering space looks approximately like $S^{1}$, there is at least one other arc of the same length with endpoints $\tilde{p}_{1}$ and $\tilde{q}_{1}$ that goes the other way around $S^{1}$.\\

The symbol $[\alpha]\in H_{1}(S_{g};\mathbb{Q})$ will be used to denote the rational homology class represented by the oriented curve $\alpha$.\\

The next step is to show that at $\tilde{x}$, for sufficiently large $M$, it is possible to find a weighted multicurve $\tilde{m}_{[\alpha]}$ in $\tilde{S}_{g(M)}$independent of choices that projects onto a weighted multicurve in the genus $g$ surface. \\

Recall that the surface corresponding to $\tilde{x}$ coarsely looks like a copy of $S^{1}$, and the hyperbolic structure is invariant under the action of the deck transformation group. Denote by $d_{F}(\tilde{x})$ the diameter of the fundamental domain obtained by cutting $S_{g}$ along $\alpha$, and $d(\tilde{x})$ the diameter of the surface corresponding to $\tilde{x}$. For generic $r$ in the range $d_{F}(\tilde{x})<r<d(x)-d_{F}(\tilde{x})$, the set $B(r,p_{1}):=\{\tilde{y}\in\tilde{S}_{g(M)}\ |\ d(\tilde{p}_{1}, \tilde{y})\leq r\}$ has boundary $\partial B(r,\tilde{p}_{1})$ consisting of 2 multicurves, $b_{1}(r,\tilde{p}_{1})$ and $b_{2}(r,\tilde{p}_{1})$. Each of these multicurves is homologous to an integer multiple of $\tilde{\alpha}$, where $\tilde{\alpha}$ is the pre-image of $\alpha$. The multicurve $\tilde{m}_{\tilde{p}_{1}}$ on $\tilde{S}_{g(M)}$ is obtained by letting $r$ vary over all generic values less that the diameter of the surface, and discarding any contractible connected components of the level sets of $r$. \\

Let $\sigma$ be a choice of generator of the deck transformation group of the cover. Recall that curves are only defined up to isotopy. For sufficiently large $r$ less than $d(x)-2d_{F}(\tilde{x})$, $b_{1}(r,\tilde{p}_{1})$ satisfies 
\begin{equation}
\label{disp}
b_{1}(r,\sigma(\tilde{p}_{1}))=b_{1}(r\pm w(r),\tilde{p}_{1})
\end{equation} where $w>0$ is a measure of the width of a fundamental domain as it does not vary much with $r$. The justification for Equation \eqref{disp} is that the level sets of distance from $\tilde{p}_{1}$ and $\sigma(\tilde{p}_{1})$ can be made arbitrarily close to parallel by making $r$ sufficiently large (this requires large $M$). When two embedded loops are sufficiently close with respect to Hausdorff distance, one is contained in the collar of the other, and the two embedded loops represent the same curve. Similarly for $b_{2}$. When $\tilde{p}_{1}$ is replaced by a different point $\tilde{p}'_{1}$ within the same fundamental domain, the same argument shows that $$b_{1}(r,\tilde{p}_{1})=b_{1}(r+w'(r),\tilde{p}'_{1})$$ for large generic $r$ less than $d(x)-d_{F}(\tilde{x})$ and some real function $w'(r)$ that does not vary much with $r$.\\

In what follows, sets of weighted curves will refer to elements of a real vector space with a basis corresponding to unoriented curves on $\tilde{S}_{g(M)}$. A curve or multicurve is a set of weighted curves with every connected component having weight one and for which curves in the set are pairwise disjoint. \\

For the point $\tilde{x}$, take an average of $\tilde{m}_{\tilde{p}_{i}}$ over all $p_{i}$ in the orbit of $p_{1}$ under the deck transformation group. For sufficiently large $M$, this average is a set of weighted curves $\tilde{m}_{[\alpha]}$ with coefficients close to 1, and a set of weighted curves with coefficients close to zero. The set of curves with coefficient close to 1 consists of multicurves $b_{1}(p_{i}, r)$ or $b_{2}(p_{i}, r)$ for which $r$ is not small or close to $d(\tilde{x})$ and $i\in \{1, 2, \ldots, M\}$, and curves with coefficient close to zero consist of multicurves $b_{1}(p_{i}, r)$ or $b_{2}(p_{i}, r)$ for which $r$ is small or close to $d(\tilde{x})$ and $i\in \{1,2,\ldots, M\}$.\\

By abuse of notation, the symbol $\tilde{m}_{[\alpha]}$ will also be used to denote the corresponding set of curves without the weights.\\

Claim: $\tilde{m}_{[\alpha]}$ projects onto a multicurve $m_{[\alpha]}$ in $S_{g}$ for which the homology class is a multiple of $[\alpha]$. Moreover, $m_{[\alpha]}$ does not depend on the choice of $p_{1}$.\\

The restriction of $\tilde{m}_{[\alpha]}$ to a fundamental domain is given by the isotopy classes of generic level sets of $r$ when measured from a distant point in the orbit of $\tilde{p}_{1}$ with $d_{F}(\tilde{x})<r<d(x)-d_{F}(\tilde{x})$. It follows that $\tilde{m}_{[\alpha]}$ is a multicurve. Also, $\tilde{m}_{[\alpha]}$ is invariant under the action of the deck transformation group. Consequently, $m_{[\alpha]}$ is a multicurve, because otherwise its pre-image $\tilde{m}_{[\alpha]}$ would have self-intersections.\\

For sufficiently large values of $r$ with $d_{F}(\tilde{x})<r<d(x)-d_{F}(\tilde{x})$, moving $\tilde{p}_{1}$ around in a fundamental domain gives level sets of $r$ that are close to parallel, and hence represent the same curves. This shows that $m_{[\alpha]}$ does not depend on the choice of $\tilde{p}_{1}$, concluding the proof of the claim.\\

Suppose $\delta'>0$ is less than the Margulis constant $\epsilon_{M}$, i.e. $\mathcal{T}_{g}^{\delta'}$ contains a neighbourhood of $\mathcal{T}_{g}^{\epsilon_{M}}$. Let $\delta=\frac{\delta'}{2}$. Denote by $D^{\delta}$ the intersection of $D$ with $\mathcal{T}_{g}^{\delta}$. The equivalence relation $x\sim y$ if $m_{[\alpha]}(x)=m_{[\alpha]}(y)$ partitions $D^{\delta}$ into blocks.\\


For $\delta$ sufficiently small, it follows from the collar lemma that the systoles in the thin part of $D$ are contained in the set $h(D)_{v}$. The assumption on $\alpha$ ensures that these systoles lift to systoles in the covering space. Choosing $\delta$ sufficiently small ensures that near the boundary of $\mathcal{T}_{g}^{\delta'}$, the blocks will be labelled by multicurves containing the systoles.\\ 

Claim: Suppose every neighbourhood of $x$ intersects several different blocks. Then there is a neighbourhood $N(x)$ of $x$ in $D^{\delta}$ with the property that $N(x)$ only intersects blocks labelled by multicurves that contain a submulticurve $m$ of $m_{[\alpha]}(x)$.\\

To prove the claim, suppose that every neighbourhood of $x$ intersects several different blocks. Let $C$ be a multicurve contained in $m_{[\alpha]}(x)$ such that $\tilde{m}_{[\alpha]}(\tilde{x})$ contains a lift of $C$ on the boundary of $B(r,\tilde{p}_{i})$ for some $i\in \mathbb{Z}$ and some large $r$ in the interval $d_{F}(\tilde{x})<r<d(\tilde{x})-d_{F}(\tilde{x})$. Any sufficiently small deck transformation group-equivariant deformation of the hyperbolic structure corresponding to $\tilde{x}$ will preserve $C$. The same holds for any other multicurve in $\tilde{m}_{[\alpha]}(\tilde{x})$ contained in a level set. This concludes the proof of the claim.\\

Starting at a point $x$ of $D$ the multicurves labelling the blocks will now be used to construct a stretch path starting at $x$. The claim will then be used to show that these stretch paths vary smoothly from point to point, giving the desired homotopy.\\

Let $\{U_{i}\}$ be an open cover of $D^{\delta}$ such that, for every $U_{i}$, there is a unique smallest dimensional block $b_{i}$ intersected by $U_{i}$, and $U_{i}$ only intersects blocks incident on $b_{i}$. It will also be assumed that the open cover contains open sets contained in top dimensional blocks. The multicurve labelling the block $b_{i}$ of $U_{i}$ will be denoted by $m_{i}$. Let $\{(\phi_{i}, U_{i})\}$ be a smooth partition of unity of $D^{\delta}$, and let $\chi(x)$ be a smooth function taking the value 1 everywhere on $\mathcal{T}_{g}^{\delta'}$, and zero on the closure of the $\delta$-thin part of $\mathcal{T}_{g}$. The partition of unity $(\phi_{i}, U_{i}), i\in \mathcal{I}$ is used to obtain a weighted multicurve
\begin{equation*}
m(x):=\sum_{i\in \mathcal{I}} \chi(x)\phi_{i}(x)m_{i}
\end{equation*}
The weighted multicurve $m(x)$ determines a stretch path $\gamma_{m(x)}:[0,1)\rightarrow \mathcal{T}_{g}$ with $\gamma_{m(x)}(0)=x$, where the curves in $m(x)$ are shortened along the stretch path at a rate proportional to the weight of the curve. As these stretch paths vary smoothly with $x$, a homotopy $\psi_{t}$ is obtained, where $\psi_{t}$ takes a point $x$ in $D$ to the point $\gamma_{m(x)}(t)$.\\

Here is where use is made of the crucial assumption that the set of curves labelling simplices in the image of the horizon map do not span $H_{1}(S_{g};\mathbb{Q})$. This assumption ensures that every curve realised as a systole somewhere on $\partial D^{\delta}$ lifts to a systole in the pre-image of $D$ in $\mathcal{T}_{g(M)}$. The parameter $\delta$ was chosen small enough that for points near the boundary of $D^{\delta}$, the blocks are labelled by multicurves containing the systoles, so the homotopy decreases the systole function at points of $D^{\delta}$ near $\partial D^{\delta}$.\\

Since $D^{\delta}$ is compact, there is a $0<T<1$ such that for $t>T$, $\psi_{t}(D^{\delta})$ is in the $\delta'$-thin part of $\mathcal{T}_{g}$. It follows from \cite{Thurston} that $\mathcal{T}_{g}$ equivariantly deformation retracts onto $\mathcal{T}_{g}^{\delta'}$. If $\mathcal{C}^{\circ}_{g}$ is embedded in the boundary of the $\delta$-thick part of $\mathcal{T}_{g}$, this then gives a homotopy of $D^{\delta}$ onto the image of a subcomplex of $\mathcal{C}^{\circ}_{g}$ with boundary given by $h(D)$.
\end{proof}

It will now be explained how Lemma \ref{oneway} provides information about the subcomplexes obtained as the images of equivariant deformation retractions of $\mathcal{P}_{g}$.

\begin{theorem}[Theorem \ref{whichsubcomplex} of the Introduction.]
For every $g\geq 2$ there is a well-rounded deformation retraction of $\mathcal{T}_{g}$ onto a CW complex of dimension $4g-5$.
\end{theorem}
\begin{proof}

This theorem is a consequence of Lemma \ref{oneway} and the construction in \cite{Tspine}. Suppose there is a locally top-dimensional cell with a dual $D$ for which $h(D)_{v}$ does not span $H_{1}(S_{g};\mathbb{Q})$. Lemma \ref{oneway} then shows that there is an homotopy of $D$ that fixes $D\cap \mathcal{T}_{g}^{\epsilon_{M}}$ and takes $D$ to an embedding of a ball disjoint from $\mathcal{P}_{g}^{X}$. This is used to show that $\mathcal{P}_{g}^{X}$ has nonempty boundary. As explained in detail in \cite{Tspine}, this boundary is used to construct a deformation retraction. This deformation retraction is equivariant, because it is defined in terms of level sets of the systole function, which is a $\Gamma_{g}$-equivariant function on $\mathcal{T}_{g}$.\\

This construction can be repeated on the resulting complex, and iterated until a complex $\mathcal{W}_{g}$ that is the image of a well-rounded deformation retraction is obtained. Only finitely many iterations are possible, because each iteration replaces an orbit of cells by an orbit cells of smaller dimension, and there are only finitely many orbits of cells.\\

The dimension of $\mathcal{W}_{g}$ cannot be less than $4g-5$, because this is the virtual cohomological dimension of $\Gamma_{g}$, \cite{Harer}. If the dimension of $\mathcal{W}_{g}$ is larger than $4g-5$, a further deformation retraction is performed as outlined below, identical to the construction in \cite{Tspine}. It follows from Remark \ref{subcomplex} that the resulting deformation retraction is well-rounded.\\

Suppose $\mathcal{W}_{g}$ has a cell of dimension dimension greater than $4g-5$. By assumption, this cell has a dual $D'$ of dimension less than $2g-1$. It was shown in \cite{Harer} that $\partial \mathcal{T}_{g}^{\epsilon_{M}}$ is homotopy equivalent to a wedge of spheres $\vee_{i}^{\infty}S^{2g-2}$. Since the intersection of $D'$ with $\partial\mathcal{T}_{g}^{\epsilon_{M}}$ has dimension less than $2g-2$, it cannot represent a nontrivial homology class in $\partial\mathcal{T}_{g}^{\epsilon_{M}}$. This means that $D'$ can be homotoped relative to its boundary out of $\mathcal{T}_{g}^{\epsilon_{M}}$. As before, this implies the complex has nonempty boundary, and an equivariant deformation retraction can be constructed.
\end{proof}

\bibliography{Spinebib2}
\bibliographystyle{plain}
\end{document}